\documentclass[reqno, 10pt,letterpaper]{amsart}

\usepackage{amsmath}
\usepackage{amssymb}
\usepackage{amsthm}
\usepackage{amscd}
\usepackage[ansinew]{inputenc}
\usepackage{cite}
\usepackage{color}
\usepackage[final]{graphicx}

\renewcommand{\epsilon}{\varepsilon}
\numberwithin{equation}{section}

\newtheoremstyle{thmlemcorr}{10pt}{10pt}{\itshape}{}{\bfseries}{.}{10pt}{{\thmname{#1}\thmnumber{ #2}\thmnote{ (#3)}}}
\newtheoremstyle{thmlemcorr*}{10pt}{10pt}{\itshape}{}{\bfseries}{.}\newline{{\thmname{#1}\thmnumber{ #2}\thmnote{ (#3)}}}
\newtheoremstyle{defi}{10pt}{10pt}{\itshape}{}{\bfseries}{.}{10pt}{{\thmname{#1}\thmnumber{ #2}\thmnote{ (#3)}}}
\newtheoremstyle{remexample}{10pt}{10pt}{}{}{\bfseries}{.}{10pt}{{\thmname{#1}\thmnumber{ #2}\thmnote{ (#3)}}}
\newtheoremstyle{ass}{10pt}{10pt}{}{}{\bfseries}{.}{10pt}{{\thmname{#1}\thmnumber{ A#2}\thmnote{ (#3)}}}

\theoremstyle{thmlemcorr}
\newtheorem{theorem}{Theorem}
\numberwithin{theorem}{section}
\newtheorem{lemma}[theorem]{Lemma}
\newtheorem{corollary}[theorem]{Corollary}
\newtheorem{proposition}[theorem]{Proposition}

\theoremstyle{thmlemcorr*}
\newtheorem{theorem*}{Theorem}
\newtheorem{lemma*}[theorem]{Lemma}
\newtheorem{corollary*}[theorem]{Corollary}
\newtheorem{proposition*}[theorem]{Proposition}
\newtheorem{problem*}[theorem]{Problem}
\newtheorem{conjecture*}[theorem]{Conjecture}

\theoremstyle{defi}

\theoremstyle{remexample}
\newtheorem{remark}[theorem]{Remark}

\theoremstyle{ass}

\newcommand{\Acal}{\mathcal{A}}

\newcommand{\Fcal}{\mathcal{F}}

\newcommand{\Ical}{\mathcal{I}}

\newcommand{\Pcal}{\mathcal{P}}

\newcommand{\Ucal}{\mathcal{U}}
\newcommand{\Vcal}{\mathcal{V}}

\newcommand{\Abb}{\mathbb{A}}

\newcommand{\Ibb}{\mathbb{I}}

\newcommand{\Pbb}{\mathbb{P}}
\newcommand{\Qbb}{\mathbb{Q}}

\newcommand{\Sbb}{\mathbb{S}}

\DeclareMathOperator{\range}{range}

\DeclareMathOperator{\Span}{span}

\DeclareMathOperator{\divergerg}{div}
\DeclareMathOperator{\curl}{curl}

\DeclareMathOperator{\rank}{rank}

\newcommand{\set}[2]{\left\{\, #1 \ \ \textup{:}\ \ #2 \,\right\}}

\newcommand{\norm}[1]{\|#1\|}

\newcommand{\normb}[1]{\bigl\|#1\bigr\|}

\newcommand{\abs}[1]{|#1|}

\newcommand{\absb}[1]{\bigl|#1\bigr|}

\newcommand{\dd}{\;\mathrm{d}}

\newcommand{\N}{\mathbb{N}}
\newcommand{\R}{\mathbb{R}}

\newcommand{\weakly}{\rightharpoonup}

\newcommand{\Lin}{\mathrm{Lin}}
\newcommand{\eps}{\epsilon}

\newcommand{\sbullet}{\begin{picture}(1,1)(-0.5,-2)\circle*{2}\end{picture}}
\newcommand{\frarg}{\,\sbullet\,}

\title[Thin-film $\Gamma$-limit in micromagnetics]{Another approach to the thin-film $\Gamma$-limit of the 
micromagnetic free energy \\in the regime of small samples}

\author{Carolin Kreisbeck}
\address{Departamento de Matem\'{a}tica and Centro de Matem\'{a}tica e Aplica\c{c}o\~{e}s, Faculdade de Ci\^{e}ncias e Tecnologia, 
Universidade Nova de Lisboa, Quinta da Torre, 2829-516 Caparica, Portugal}
\email{carolink@andrew.cmu.edu}

\begin{document}

\linespread{1.1}


\begin{abstract}The asymptotic behavior of the micromagnetic free energy governing a ferromagnetic film is 
studied as its thickness gets smaller and smaller compared to its cross section. Here the static Maxwell equations 
are treated as a Murat's constant-rank PDE constraint on the energy functional. In contrast to previous work this 
approach allows to keep track of the induced magnetic field without solving the magnetostatic equations.
In particular, the mathematical results of Gioia and James~[\textit{Proc.\ R.\ Soc.\ Lond.\ A} 453 (1997), pp. 213--223] 
regarding convergence of minimizers are recovered by giving a characterization of the corresponding $\Gamma$-limit.
\vspace{8pt}

 

\vspace{8pt}

\noindent\textsc{Date:} \today.
\end{abstract}

\maketitle


\section{Introduction}

Over the last twenty years there has been tremendous scientific progress in the research on thin-film devices pushing 
forward technology and leading to important industrial applications. 
By a thin film one understands a layer of material whose thickness 
ranges between a fractional amount of a nanometer and a couple of micrometers. When using thin films in computer data 
storage media and solar cells, a deep understanding of their ferromagnetic properties is of great 
importance \cite{Ohr02, HS09}. A widely used mathematical procedure to achieve exactly that takes the theory of 
micromagnetics for bulk bodies \cite{Bro62, Bro63} as a starting point and derives 
reduced theories capable of capturing the specific features of thin material layers by means of dimension reduction techniques.

It was in \cite{GJ97} that for the first time authors studied convergence of minimizers of the micromagnetic energy 
on a film whose thickness gets smaller and smaller relative to its cross section. 
In this article we follow the same approach regarding scaling, which corresponds to considering the regime of
very small film samples, but we employ an equivalent formulation in the sense of \cite{DeS93}.~The latter 
illustrates that micromagnetism is actually one of the examples, where the mathematical modeling of physical phenomena
within a variational formulation leads to functionals that do not simply depend on gradient fields. Instead one faces more
intricate partial differential constraints that involve an interaction of divergence- and curl-free vector fields. Precisely, 
the static Maxwell (or magnetostatic) equations, which govern the relation between the magnetization $\bar{m}$ of a ferromagnetic body 
occupying a bounded domain $\Omega\subset \R^3$ and its induced magnetic field $\bar{h}:\R^3\to \R^3$, read
\begin{align}\label{Maxwell}
\divergerg \bigl(\bar{m}+\bar{h}\bigr)=0\qquad\text{ in $\R^3$,}\\
\qquad \curl \bar{h}=0\qquad\text{ in $\R^3$},\nonumber
\end{align}
where $\bar{m}:\Omega\to\R^3$ is identified with its trivial extension to the whole space by zero.
In the literature this mathematical difficulty is commonly tackled by explicitly solving
or proving existence of solutions to the magnetostatic equations in their weak formulation and expressing $\bar{h}$ in terms of $\bar{m}$. 
Our idea is to keep both $\bar{m}$ and $\bar{h}$ as variables and work directly with the PDE constraint imposed by \eqref{Maxwell}. 
The main result of this paper is a rigorous $\Gamma$-convergence based $3$d--$2$d dimension reduction for the constrained micromagnetic functional.
Our approach allows us to keep track during the limit process not only of the magnetization, but at the same time also of the induced field. 


\section{Formulation of the problem and statement of the main result}\label{sec:problem}
Let $\Omega_\eps=\omega\times (0,\eps)$ model a ferromagnetic body of thickness $\eps>0$ with cross section $\omega\subset \R^2$, 
wlog $\abs{\omega}=1$.
The free energy per unit volume that emerges in the theory of micromagnetism \cite{Bro63, LLP84} is given by
\begin{align*}
E_\eps\left[\bar{m}, \bar{h}\right]=
\left\{\begin{array}{cl}
\displaystyle\frac{1}{\eps} \int_{\Omega_\eps} \alpha \,\abs{\nabla \bar{m}}^2 + \varphi(\bar{m})\dd{y} 
+ \frac{1}{2} \int_{\R^3}\abs{\bar{h}}^2 \dd{y}, &\text{if $(\bar{m},\bar{h})\in \Vcal_\eps$},\\
\infty, & \text{ otherwise\,.}
\end{array}\right.
\end{align*}
Here $\alpha>0$ is a material constant and $\varphi:\R^3\to \left.\left[0,\infty\right.\right)$ is a continuous, 
even function featuring crystallographic symmetry. Further,
\begin{align*}
\Vcal_\eps&=\Bigl\{(\bar{m},\bar{h})\in W^{1,2}(\Omega_\eps;\R^3)\times L^2(\R^3;\R^3):\\ 
&\qquad\qquad\qquad\qquad\qquad\qquad\qquad\Acal^{\text{\rm mag}} \Bigl(\begin{array}{c}
    \bar{m} \\
    \bar{h} 
  \end{array}\Bigr)=0 \text{ in $\R^3$}, \abs{\bar{m}}=m_s \text{ in $\Omega_\eps$}\Bigr\}.
\end{align*}

The first order PDE constraint in the definition of $\Vcal_\eps$, namely
\begin{align*}
\Acal^{\text{\rm mag}}\left(\begin{array}{c}
    \bar{m} \\
    \bar{h} 
\end{array}\right):=\left( \begin{array}{c|c}
    \divergerg & \; \divergerg \\ \hline
    0   & \; \curl
  \end{array}\right)\left(\begin{array}{c}
    \bar{m} \\
    \bar{h} 
\end{array}\right)=\left(\begin{array}{c}
    \divergerg (\bar{m}+\bar{h}) \\
    \curl \bar{h}
\end{array}\right)= 0,
\end{align*} 
conveys the magnetostatic equations. 
The operator $\curl$ is supposed to be interpreted as $\curl=\nabla\times$, i.e.\ 
\begin{align*}
\curl \bar{h}= \bigl(\partial_2 \bar{h}_3-\partial_3\bar{h}_2,\, \partial_3 \bar{h}_1-\partial_1 \bar{h}_3,\, 
\partial_1 \bar{h}_2-\partial_2 \bar{h}_1\bigr)^T.
\end{align*}
Notice that throughout this work all occurring differential operators and partial derivatives are to be understood in the 
sense of distributions, for example $\divergerg(\bar{m}+\bar{h})=0$ in $\R^3$ means $-\int_{\R^3}(\bar{m}+\bar{h})\cdot \nabla \phi\dd{y}=0$ 
for all test functions $\phi\in C_c^\infty(\R^3)$.

Physically speaking, the nonconvex constraint $\abs{\bar{m}}=m_s$ in $\Omega_\eps$ (depending on the regularity of $\bar{m}$, 
this equality may only be fulfilled pointwise a.e.\ in $\Omega_\eps$) encodes the fundamental assumption that the body is 
locally saturated with saturation magnetization $m_s>0$. The second term in the definition of $E_\eps$ is the anisotropy energy, which 
penalizes magnetizations varying from special directions within the crystal lattice of the ferromagnet. The latter are called 
the easy axes of magnetization. The contribution of exchange energy is captured by the first term of $E_\eps$. It results from a force 
tending to align magnetic moments of neighboring atoms and therefore favors regions of constant magnetization. The third summand in 
$E_\eps$ is an integral over the whole space $\R^3$ modeling the energy of the magnetic field $\bar{h}$ induced by $\bar{m}$. These 
three energy components impose competing requirements on the magnetization and minimizers of $E_\eps$ may form interesting structures. In the 
case of bulk bodies one observes Weiss domains separated by thin Bloch walls. 
For more details on the physical motivation and interpretation of the nonlocal and 
nonconvex energy $E_\eps$ see for example \cite{Vis85, JK90, JM94, ABV91} and the references therein.

As originally stated in \cite{Tar79} and further discussed in~\cite{FM99}, $\Acal^{\text{\rm mag}}$ is a first order differential 
operator that meets Murat's constant-rank property \cite{Mur81}, i.e.\  the symbol $\Abb_{\Acal^{\text{\rm mag}}}$ of $\Acal^{\text{\rm mag}}$ satisfies
\begin{align}\label{constantrank}
\rank \,\Abb_{\Acal^{\text{\rm mag}}}(\xi) = const. \qquad \text{for all $\xi \in \Sbb^2$}.
\end{align}
Indeed, $\Abb_{\Acal^{\text{\rm mag}}}(\xi) \in \Lin(\R^3\times \R^3,\,\R\times\R^3)$ and it holds 
\begin{align*}\ker \Abb_{\Acal^{\text{\rm mag}}}(\xi) &= \set{(a,b)\in \R^3\times \R^3}{\xi \cdot (a+b)=0, \ \xi\times b=0}\\ 
&= \set{(a,b)\in \R^3\times \R^3}{b=\lambda \xi,\ \xi\cdot a=-\lambda, \ \lambda\in \R}.
\end{align*}
So $\rank \Abb_{\Acal^{\text{\rm mag}}}(\xi)=6-\dim \ker \Abb_{\Acal^{\text{\rm mag}}}(\xi)=3$ for all $\xi\in \Sbb^2$. Hence, 
the problem we are interested in can be studied in 
the more abstract context of dimension reduction for functionals on $\Acal$-free vector fields. 
The work on variational problems within the $\Acal$-free framework can be traced back to~\cite{Dac82} and was advanced
by Fonseca and M{\"u}ller~\cite{FM99}, who came up with the notion of $\Acal$-quasiconvexity (in its modern sense) and studied 
lower semicontinuity of  functionals involving integrands with this property. Since then, a lot of papers investigating for 
instance relaxation, homogenization and Young measures in the $\Acal$-free setting have been published~\cite{BFL00, FK10, FLM04, FM99}. 
The first article to cover $3$d--$2$d asymptotic analysis in such generality is~\cite{KR11}, while the special case of thin-film 
limits for gradient dependent problems ($\Acal=\curl$) has been treated before (see for instance~\cite{LR95, LR00, FFL07}). 
For a recent result in the context of
functionals on solenoidal vector fields ($\Acal=\divergerg$) we refer to \cite{Kro10}. In fact, \cite{KR11} provides the technical basis for the 
work presented in the following.
 
To obtain a variational problem on the fixed domain $\Omega_1=\omega \times (0,1)$, we apply the standard parameter rescaling,
\begin{align*}
x=(x',x_d)=(y', \eps^{-1} y_d) 
\end{align*} 
for $y=(y_1, \ldots, y_{d-1},y_d)=(y', y_d)\in \Omega_\eps$, and we set $m(x)=\bar{m}(y)=\bar{m}(x',\eps x_d)$ and 
$h(x)=\bar{h}(y)=\bar{h}(x',\eps x_d)$. Then $E_\eps$ transforms into 
\[
F_\eps[m,h]=
\left\{\begin{array}{cl}
\displaystyle\int_{\Omega_1} \alpha \,\abs{\nabla_\eps m}^2 + \varphi(m)\dd{x} + \frac{1}{2} \int_{\R^3}\abs{h}^2 \dd{x}, &\text{if $(m,h)\in \Ucal_\eps$},\\
\infty, & \text{ otherwise\,,}
\end{array}\right.
\]
where 
\begin{align*}
\Ucal_\eps&=\Bigl\{(m,h)\in W^{1,2}(\Omega_1;\R^3)\times L^2(\R^3;\R^3):\\ 
&\qquad\qquad\qquad\qquad\qquad\qquad\qquad\Acal^{\text{\rm mag}}_\eps\left(\begin{array}{c}
    \bar{m} \\
    \bar{h} 
\end{array}\right)=0 \text{ in $\R^3$}, \abs{m}=m_s \text{ in $\Omega_1$}\Bigr\}.
\end{align*} 
The rescaled versions of the operators $\nabla=(\nabla', \partial_3)^T=(\partial_1, \partial_2, \partial_3)^T$ and $\Acal^{\text{\rm mag}}$ are given by 
\begin{align}\label{def:nabla_eps}
\nabla_\eps=\Bigl(\nabla',\, \frac{1}{\eps}\partial_3\Bigr)^T \qquad\text{and}\qquad \Acal^{\text{\rm mag}}_\eps=\Abb_{\Acal^{\text{\rm mag}}}(\nabla_\eps),
\end{align}
respectively. Accordingly, we define
\begin{align*}
\divergerg_\eps= \nabla_\eps \cdot \qquad\text{and}\qquad \curl_\eps=\nabla_\eps\times.
\end{align*}
In view of these definitions the main result is the following:

\begin{theorem}\label{theo:mainresult_Gamma}
The $\Gamma$-limit of $F_\eps$ for $\eps\to 0^+$ with respect to weak convergence in $W^{1,2}(\Omega_1;\R^3)\times L^2(\R^3;\R^3)$ exists and 
is represented as
\[
F_0[m,h]=
\left\{\begin{array}{cl}
\displaystyle\int_{\Omega_1} \alpha\, \abs{\nabla' m}^2 + \varphi(m)\dd{x}+ \frac{1}{2}\int_{\R^3}\abs{h}^2\dd{x}, &\text{if $(m,h)\in \Ucal_0$},\\
\infty, & \text{ otherwise\,,}
\end{array}\right.
\] 
with 
\begin{align*}
\Ucal_0 &=\Bigl\{(m,h)\in W^{1,2}(\Omega_1;\R^3)\times L^2(\R^3;\R^3)\,:\\ 
&\qquad\qquad\qquad\qquad \Acal_0^{\text{\rm mag}}\left(\begin{array}{c}
    m \\
    h 
\end{array}\right) =0 \text{ in $\R^3$},\, \partial_3 m=0 \text{ in $\Omega_1$},\,\abs{m}=m_s \text{ in $\Omega_1$}\Bigr\}
\end{align*}
and $\Acal_0^{\text{\rm mag}}$ defined through
\begin{align*}
 \Acal^{\text{\rm mag}}_0\left(\begin{array}{c}
    m \\
    h 
\end{array}\right):=\left(\begin{array}{c}
  \partial_3(m_3+h_3)\\
 -\partial_3 h_2\\ 
  \partial_3 h_1\\  
  \partial_1 h_2-\partial_2 h_1
\end{array}\right).
\end{align*}
Moreover, compactness holds in the weak topology of $W^{1,2}(\Omega_1;\R^3)\times L^2(\R^3;\R^3)$. 
\end{theorem}

As we discuss in Section~\ref{sec:comparisonGJ97} the result of Theorem~\ref{theo:mainresult_Gamma} is in complete agreement with the 
limiting micromagnetic energy derived in~\cite{GJ97}. 

In \cite{KR11}, dimension reduction within the general $\Acal$-free setting is investigated. One purpose of Theorem~\ref{theo:mainresult_Gamma} is to 
illustrate the power of the concepts and tools developed there by means of another physically relevant example. 
(An immediate first application to bending of thin films in nonlinear elasticity is studied in~\cite[Section~5]{KR11}.)

Since the functionals $F_\eps$ contain first order derivatives of $m$ and involve a nonconvex constraint, they do not exactly fit into 
the context of \cite{KR11}. (The issue of $F_\eps$ being defined on functions on the whole space can be overcome by replacing Fourier 
series with Fourier transforms in the proofs. However, it is not clear how to extend that theory to mixed-order differential operators.)
Therefore, we consider $F_\eps$ as split into a part that is convex in the derivatives of 
$m$ and one that is in line with \cite{KR11}. 
When it comes to proving the upper bound, the crucial step is to exploit a tool introduced in \cite{KR11}, which yields convergence of 
the symbols of $\Acal^{\text{\rm mag}}_\eps$ for $\eps\to 0^+$. This provides a ``first candidate" for the recovery sequence which, however, 
lacks the necessary regularity and fails to meet the nonconvex constraint imposed by the requirement of local saturation. To handle this matter, 
we modify the sequence by choosing the magnetizations to be constant and by adjusting the exterior fields with the help of projection operators onto 
$\curl_\eps$-free fields.

In the next two sections we give the detailed proof of Theorem~\ref{theo:mainresult_Gamma} by showing separately the 
required upper and lower bounds. We remark that throughout this work we use generalized sequences with index $\eps>0$, 
like $(u_\eps)_\eps$, by which we refer to any sequence $(u_{\eps_j})_j$ with $\eps_j\to 0^+$ as $j\to \infty$.


\section{Proof of compactness and the lower bound}

We begin by proving the following compactness result, which is essentially based on the coercivity of the micromagnetic free energy. 
Notice that extracted subsequences are not relabeled in the sequel.

\begin{proposition}[Compactness]\label{prop:compactness}
Let $\eps_j\to 0^+$ for $j\to \infty$. Further assume that $(m_{\eps_j}, h_{\eps_j})_j\subset W^{1,2}(\Omega_1;\R^3)\times L^2(\R^3;\R^3)$ is
a bounded energy sequence for $F_{\eps_j}$, i.e.\ \begin{align*}
F_{\eps_j}[m_{\eps_j}, h_{\eps_j}]\leq C<\infty \qquad\text{ for all $j\in \N$}.
\end{align*}
Then there exists a subsequence $(m_{\eps_j}, h_{\eps_j})_j$ and $(m,h)\in  W^{1,2}(\Omega_1;\R^3)\times L^2(\R^3;\R^3)$ such that 
\begin{align}\label{convergence_mepsheps}
m_{\eps_j} &\weakly m \qquad \text{in $W^{1,2}(\Omega_1;\R^3)$,}\\
h_{\eps_j} &\weakly h \qquad \text{ in $L^2(\R^3;\R^3)$}\nonumber
\end{align}
for $j\to \infty$. Moreover, it holds that $(m,h)\in \Ucal_0$.
\end{proposition}

\begin{proof}
In view of the constraint $|m_{\eps_j}| = m_s$ in $\Omega_1$ and the fact that 
\begin{align}\label{bound_gradepsm}
\norm{\nabla_{\eps_j} m_{\eps_j}}_{L^2(\Omega_1;\R^{3\times 3})}\leq C<\infty\qquad\text{ for all $j\in \N$},
\end{align}
one infers that $\norm{m_{\eps_j}}_{W^{1,2}(\Omega_1;\R^{3\times 3})}$ is bounded uniformly with respect to $j$,
which implies the existence of a subsequence $(m_{\eps_j})_j$ and a function $m\in W^{1,2}(\Omega_1;\R^3)$ such that 
$m_{\eps_j}\weakly m$ in $W^{1,2}(\Omega_1;\R^3)$. By compact embedding we find (after passing to a subsequence) 
that $m_{\eps_j}\to m$ pointwise a.e.\ in $\Omega_1$ as $j\to \infty$, so that $|m|= m_s$ in $\Omega_1$. Recalling 
the definition of $\nabla_{\eps}$ in $\eqref{def:nabla_eps}$, we conclude from $\eqref{bound_gradepsm}$ that 
$\partial_3 m_{\eps_j}\to 0$ in $L^2(\Omega_1;\R^3)$. Thus, $\partial_3 m=0$ in $\Omega_1$ by uniqueness of the limit. 
Since the induced energy contribution of $F_{\eps_j}[m_{\eps_j}, h_{\eps_j}]$ is bounded, one can extract a subsequence 
of $(h_{\eps_j})_j$ satisfying
\begin{align*}
h_{\eps_j}\weakly h \qquad \text{in $L^2(\Omega_1;\R^3)$} 
\end{align*}
for some $h\in L^2(\Omega_1;\R^3)$.
The expression $\curl_{\eps_j} h_{\eps_j}=0$ in $\R^3$ is equivalent to 
\vspace{-0.1cm}
\begin{align*}
 \partial_2 (h_{\eps_j})_3 - 1/{\eps_j}\,\partial_3 (h_{\eps_j})_2 &=0 \qquad\text{ in $\R^3$},\\
 1/{\eps_j}\, \partial_3 (h_{\eps_j})_1 - \partial_1(h_{\eps_j})_3 &=0 \qquad\text{ in $\R^3$},\\
 \partial_1(h_{\eps_j})_2-\partial_2(h_{\eps_j})_1 &=0\qquad\text{ in $\R^3$}.
\end{align*}
When passing to the limit $j\to \infty$, it follows that $\partial_3 h_2=\partial_3 h_1=0$ and $\partial_1 h_2=\partial_2 h_1$ in $\R^3$. 
Finally we exploit $\divergerg_{\eps_j}(m_{\eps_j} + h_{\eps_j})=0$ in $\R^3$ to derive
\begin{align*}
\partial_3(m_3+h_3)=0 \qquad\text{in $\R^3$}.
\end{align*}
Thus, $(m,h)\in \Ucal_0$.
\end{proof}
\begin{remark}\label{remark}
If $(m_{\eps_j}, h_{\eps_j})_{\eps_j}$ is a sequence of minimizers for $F_{\eps_j}$, the convergence in \eqref{convergence_mepsheps}
can be shown to be strong in $W^{1,2}$ and $L^2$, respectively (compare \cite[Theorem~4.1]{GJ97}).
\end{remark}

Next we give the proof of the lower bound.
Let $m_\eps \weakly m$ in $W^{1,2}(\Omega_1;\R^3)$ and $h_\eps \weakly h$ in $L^2(\R^3;\R^3)$ as $\eps$ tends to zero. 
Consider any $\eps_j\to 0^+$ for $j\to \infty$ and assume that
\begin{align*}
\liminf_{j\to \infty}F_{\eps_j}[m_{\eps_j}, h_{\eps_j}]=\lim_{j\to \infty}F_{\eps_j}[m_{\eps_j}, h_{\eps_j}]<\infty,
\end{align*}
otherwise the corresponding liminf-inequality is immediate.
Then $(m_{\eps_j}, h_{\eps_j})_j\subset \Ucal_{\eps_j}$ is of bounded energy and $(m,h)\in \Ucal_0$ by 
Proposition~\ref{prop:compactness}. In view of the compact embedding $W^{1,2}(\Omega_1;\R^3)\hookrightarrow L^2(\Omega_1;\R^3)$ 
we obtain a subsequence with $m_{\eps_j}\to m$ pointwise a.e.\ in $\Omega_1$, so that by the continuity of $\varphi$, the weak lower semicontinuity 
of the $L^2$-norm and Fatou's lemma,
\begin{align*}
\lim_{j\to \infty}F_{\eps_j}[m_{\eps_j}, h_{\eps_j}]&\geq \lim_{j\to \infty}\int_{\Omega_1} \alpha \,\abs{\nabla' m_{\eps_j}}^2 
+ \varphi(m_{\eps_j})\dd{x} + \frac{1}{2} \int_{\R^3}\abs{h_{\eps_j}}^2 \dd{x}\nonumber\\
&\geq \int_{\Omega_1} \alpha \,\abs{\nabla' m}^2 + \varphi(m)\dd{x} + \frac{1}{2} \int_{\R^3}\abs{h}^2 \dd{x}.
\end{align*}
Hence, \begin{align*}\liminf_{\eps\to 0^+}F_\eps[m_\eps, h_\eps]\geq F_0[m,h],\end{align*} which is the liminf-inequality.


\section{Construction of a recovery sequence}

This section is based on arguments involving operators of the following form: 
For given matrices $A^{(1)}, \ldots, A^{(d)}\in\R^{l \times n}$ let $\Acal$ be the linear partial differential operator of first order defined through
\begin{equation}\label{operator_A}
  \Acal:= \sum_{k=1}^d A^{(k)} \partial_k.
\end{equation}
Then the symbol $\Abb_\Acal$ of $\Acal$ is given by
\[
\Abb_\Acal(\xi) := \sum_{k=1}^d A^{(k)} \xi_k,  \qquad \xi \in \R^d.
\]
The essential assumption on $\Acal$ is Murat's constant-rank condition \cite{Mur81, FM99}, precisely
\begin{align*}
\rank \,\Abb_\Acal(\xi) = const. \qquad \text{for all $\xi \in \Sbb^{d-1}$}.
\end{align*}
By \cite[Lemma~2.2]{KR11} the operators $\Acal_\eps:=\Abb_\Acal(\nabla_\eps)$ are of constant rank for all $\eps>0$ provided $\Acal$ has the same property.

As established in Section~\ref{sec:problem} (see $\eqref{constantrank}$) $\Acal^{\text{\rm mag}}$ fits into the framework described above with 
\begin{align*}
 d=3,\ n=6,\ l=4 \quad \text{and}\quad u=\Bigl(\begin{array}{c}m\\h\end{array}\Bigr).
\end{align*}

So, after having proved some technical tools for general constant-rank operators we will be able to apply these findings to the context of micromagnetics.

The next theorem is a modification of \cite[Theorem~2.7]{KR11}, which is formulated for $L^p$-functions on the torus, 
for $L^2$-functions on the whole space. The important issue in comparison with \cite[Lemma~2.14]{FM99} is to obtain constants independent of
$\eps$. A first comment in this direction is made in \cite{Kro10}, where projection operators onto $\divergerg$-free fields are investigated.

\begin{lemma}[Projection onto $\Acal_\eps$-free fields]\label{theo:projection}
Suppose $\Acal$ is a constant-rank operator as defined in $\eqref{operator_A}$ and $\eps\in (0,1)$. 
Then there exist bounded operators $\Pcal_{\Acal_\eps}:L^2(\R^d;\R^n)\to L^2(\R^d;\R^n)$ with the following properties:
\begin{enumerate}
\item[\it (i)] $(\Pcal_{\Acal_\eps} \circ \Pcal_{\Acal_\eps}) u = \Pcal_{\Acal_\eps} u$ for all $u\in L^2(\R^d;\R^n)$.
 \item[\it (ii)] $(\Acal_\eps \circ \Pcal_{\Acal_\eps}) u=0$ in $\R^d$ for all $u\in L^2(\R^d;\R^n)$.
 \item[\it (iii)] The operators $\Pcal_{\Acal_\eps}$ are uniformly bounded with respect to $\eps$, i.e.
\[
  \qquad \normb{\Pcal_{\Acal_\eps} u}_{L^2(\R^d;\R^n)} \leq  C\,\norm{u}_{L^2(\R^d;\R^n)}
\]
for all $u\in L^2(\R^d;\R^n)$ with a constant $C>0$ independent of $\eps$.
\item[\it (iv)] There exists a constant $C>0$ such that
\[
  \qquad \normb{u-\Pcal_{\Acal_\eps} u }_{L^2(\R^d;\R^n)}\leq C\, \normb{\Acal_\eps u}_{W^{-1,2}(\R^d;\R^l)}
\]
for all $u\in L^2(\R^d;\R^n)$ and all $\eps\in (0,1)$.
\end{enumerate}
\end{lemma}

\begin{proof} 
In the sequel we employ the common notation $\Fcal$ to refer to the Fourier transform and use $\Fcal^{-1}$ for its inversion.

For $\xi\in \R^d\setminus \{0\}$ let $\Pbb_{\Acal_\eps}(\xi)$ be the orthogonal projector onto $\ker \Abb_{\Acal_\eps}(\xi)\subset \R^n$. So
the mapping $\Pbb_{\Acal_\eps} \colon \R^d \setminus \{0\} \to \Lin(\R^n;\R^n)$ is $0$-homogeneous and smooth. The fact that the operator norm of 
$\Pbb_{\Acal_\eps}(\xi)$ is equal to $1$ for all $\xi\in \R^d\setminus\{0\}$ and $\eps>0$ 
renders $\Pcal_{\Acal_\eps}:L^2(\R^d;\R^n)\to L^2(\R^d;\R^n)$ defined by
\begin{align}\label{def:projector_P}
\Pcal_{\Acal_\eps} u = \Fcal^{-1} \bigl(\Pbb_{\Acal_\eps}(\frarg) \Fcal u \bigr), \qquad u\in L^2(\R^d;\R^n),
\end{align}
a continuous operator satisfying the estimate 
\begin{align*}
 \normb{\Pcal_{\Acal_\eps} u}_{L^2(\R^d;\R^n)}\leq \, \norm{u}_{L^2(\R^d;\R^n)}
\end{align*}
for all $u\in L^2(\R^d;\R^n)$. This proves $(iii)$. 
The properties $(i)$ and $(ii)$ are an immediate consequence of the structure of $\Pbb_{\Acal_\eps}$ together with $\eqref{def:projector_P}$. 

In order to show $(iv)$ let $\Qbb_{\Acal_\eps}(\xi) \in \Lin(\R^l;\R^n)$ with $\xi \in \R^d\setminus\{0\}$ and $\eps>0$ be given by
\[
  \Qbb_{\Acal_\eps}(\xi)v = \begin{cases}
    z-\Pbb_{\Acal_\eps}(\xi)z  &\text{for $v\in \range \Abb_{\Acal_\eps}(\xi)$ with
      $v=\Abb_{\Acal_\eps}(\xi)z$, $z\in \R^n$}, \\
    0 & \text{for } v\in \bigl(\range \Abb_{\Acal_\eps}(\xi)\bigr)^\perp.
  \end{cases}
\]
Then, $\Qbb_{\Acal_\eps}:\R^d\setminus\{0\}\to \Lin(\R^l;\R^n)$ is homogeneous of degree $-1$ and smooth. 
Notice that the smoothness of both $\Qbb_{\Acal_\eps}$ and $\Pbb_{\Acal_\eps}$ 
rests upon the constant-rank property of $\Acal_\eps$ (compare \cite[Proposition~2.7]{FM99}). Besides,
$\Qbb_{\Acal_\eps}(\frarg /\abs{\frarg})$ is bounded in the $L^\infty$-norm uniformly with respect to $\eps$. 
This can be seen as follows. Since $\Abb_{\Acal_\eps}(\eta)=\Abb_\Acal(\eta_\eps)$ for all $\eta\in \R^d$ with 
$\eta_\eps:=(\eta_1,\ldots, \eta_{d-1}, 1/\eps\,\eta_d)^T$, 
we may argue for $\eps\in (0,1)$ that
\begin{align*}
& \sup_{\eta\in \Sbb^{d-1}}\norm{\Qbb_{\Acal_\eps} (\eta)}_{\Lin(\R^l;\R^n)} 
= \sup_{\eta\in \Sbb^{d-1}} \norm{\Qbb_\Acal(\eta_\eps)}_{\Lin(\R^l;\R^n)}\\
&\qquad\leq \sup_{\xi\in \R^d, \abs{\xi}\geq 1} \norm{\Qbb_\Acal (\xi)}_{\Lin(\R^l;\R^n)}
= \sup_{\eta\in \Sbb^{d-1}, \alpha\geq 1} \alpha^{-1} \norm{\Qbb_\Acal(\eta)}_{\Lin(\R^l;\R^n)}\\
& \qquad\leq\sup_{\eta\in \Sbb^{d-1}}\norm{\Qbb_\Acal (\eta)}_{\Lin(\R^l;\R^n)} <\infty,
 \end{align*}
where the second equality is a consequence of the $(-1)$-homogeneity of $\Qbb_\Acal$. The final estimate 
results from $\Qbb_\Acal$ being smooth on the unit sphere $\Sbb^{d-1}$, which is a compact subset of $\R^d$.

Consequently, for $w_\eps\in L^2(\R^d;\R^l)$ given through $\Fcal w_\eps= \abs{\frarg}^{-1} \Abb_{\Acal_\eps}(\frarg) \Fcal u$ it holds that
\begin{align}\label{projection_estimate}
\normb{\Fcal^{-1}\bigl(\Qbb_{\Acal_\eps}(\frarg /\abs{\frarg})\Fcal w_\eps\bigr)}_{L^2(\R^d;\R^n)} &\leq C\, \norm{w_\eps}_{L^2(\R^d;\R^l)}\\
&\leq C\, \normb{\Acal_\eps u}_{W^{-1,2}(\R^d;\R^l)}.\nonumber
\end{align}
The last inequality holds by the definition of the $W^{-1,2}$-norm. 
Using the properties of $\Qbb_{\Acal_\eps}, \Pbb_{\Acal_\eps}$ and the linearity of $\Abb_{\Acal_\eps}$ yields
\begin{align*}
\Qbb_{\Acal_\eps}(\frarg /\abs{\frarg})\Fcal w_\eps&
=\Qbb_{\Acal_\eps}(\frarg /\abs{\frarg})\Abb_{\Acal_\eps}(\frarg /\abs{\frarg})\Fcal u \\&= \Fcal u-\Pbb_{\Acal_\eps}(\frarg/\abs{\frarg})\Fcal u=\Fcal u -\Pbb_{\Acal_\eps}(\frarg)\Fcal u.
\end{align*}
In view of $\eqref{projection_estimate}$ and $\eqref{def:projector_P}$ this proves $(iv)$. 
\end{proof}

\begin{remark}
For $p\in(1, \infty)$ the statement of Lemma~\ref{theo:projection} is still true.
However, the line of reasoning is more involved due to the fact that the Fourier inversion formula does not hold in general. 
In analogy to the proof of \cite[Theorem~2.7]{KR11} (with Fourier series replaced by Fourier transforms, where necessary), 
this difficulty can be overcome by using approximation via smooth functions in conjunction with
Mihlin's Multiplier Theorem (see for instance \cite[Theorem~5.2.7]{Gra08}) and a scaling argument for Fourier multipliers.
\end{remark}

As an essential tool towards the construction of a recovery sequence, we show the following analog of \cite[Proposition~4.1]{KR11} 
within the setting of functions defined on the whole space. Naturally, an extension property in the sense of
\cite[Assumption A3]{KR11} is not needed. The proof follows closely along the lines of \cite{KR11}, 
but is substantially easier, since one of the two relevant terms is forced to vanish here.

\begin{proposition}\label{prop:recov_1}
Let $\Acal$ be a constant-rank operator as in $\eqref{operator_A}$ such that the number of non-zero rows of the matrix $A^{(d)}$ is 
equal to the rank of $A^{(d)}$. Further suppose $u \in L^2(\R^d;\R^n)$ satisfies $\Acal_0 u = 0$ in $\R^d$, where 
\begin{align*}
 \Acal_0  &:= \left( \left\{ \begin{aligned}
      &[A^{(d)}]^i \partial_d,               && \text{if $[A^{(d)}]^i \neq 0$,} \\
      &\sum_{k=1}^{d-1} [A^{(k)}]^i \partial_k,  && \text{if $[A^{(d)}]^i = 0$}
    \end{aligned} \right\} \right)^{i = 1,\ldots,l}\nonumber
\end{align*}
and $[A^{(k)}]^i$ denotes the $i$th row of $A^{(k)}$.
Then there exists a sequence $(u_\eps)_{\eps } \subset L^2(\R^d;\R^n)$ with 
$\Acal_{\eps} u_\eps = 0$ in $\R^d$ and $u_\eps \to u$ in $L^2(\R^d;\R^n)$ for $\eps\to 0^+$.
\end{proposition}

For the proof of this proposition the auxiliary ``symbol'' 
\begin{align*}
  \tilde{\Abb}_0(\xi) :=  \left( \left\{ \begin{aligned}
      &[A^{(d)}]^i \xi_d,              && \text{if $[A^{(d)}]^i\xi_d \neq 0$,} \\
      &\sum_{k=1}^{d-1} [A^{(k)}]^i \xi_k, && \text{if $[A^{(d)}]^i\xi_d = 0$}
    \end{aligned} \right\} \right)^{i = 1,\ldots,l},\qquad \xi \in \R^d,
\end{align*}
will be needed.
By $\tilde{\Pbb}_0(\xi)$ we denote the orthogonal projection onto $\ker \tilde{\Abb}_0(\xi)\subset \R^n$.
The symbol  $\Abb_{\Acal_0}$ coincides with $\tilde{\Abb}_0$ outside the hyperplane where $\xi_d = 0$, in formulas
\begin{align}\label{def:A_tilde}
\tilde{\Abb}_0(\xi)=\Abb_{\Acal_0}(\xi) \qquad \text{ for all $\xi\in \R^d$ with $\xi_d\neq 0$.}
\end{align}
Note that in contrast to $\Abb_{\Acal_0}$, the expression $\tilde{\Abb}_0$ is not the symbol of a 
constant-coefficient partial differential operator (see~\cite[Remark~4.2]{KR11}).

Actually, $\tilde{\Abb}_0$ turns out to characterize the limit behavior of the symbols $\Abb_{\Acal_\eps}$ as $\eps$ tends to zero. 
The exact result is formulated in the next lemma, which was proven in \cite{KR11} and is repeated here for the readers' convenience.

\begin{lemma}[{\cite[Lemma 4.3]{KR11}}\,] \label{lem:P_eps_conv}
If $\Acal$ meets the assumptions of Proposition~\ref{prop:recov_1}, the symbols $\Abb_{\Acal_\eps}$ converge to 
$\tilde{\Abb}_0$ as $\eps\to 0^+$ in the sense that $\Pbb_{\Acal_\eps}(\xi)v \to \tilde{\Pbb}_0(\xi)v$ for all $\xi \in \R^d$ and all $v\in \R^n$.
\end{lemma}

\begin{proof}[Proof of Proposition~\ref{prop:recov_1}]
The fact that $u$ is $\Acal_0$-free implies 
\begin{align}\label{eq:A0-free}
  \Abb_{\Acal_0}(\xi) \bigl(\Fcal u\bigr)(\xi) = 0  \qquad \text{for a.e.\ $\xi \in \R^d$.}
\end{align}
Now we split $\Fcal u$ into
\[
  \bigl(\Fcal u\bigr)^{(1)}(\xi) := \begin{cases}
                      \bigl(\Fcal u\bigr)(\xi)  & \text{if $\xi_d \neq 0$,} \\
                      0            & \text{if $\xi_d = 0$,}
                    \end{cases}
  \quad\text{and}\quad
  \bigl(\Fcal u\bigr)^{(2)}(\xi) := \begin{cases}
                      0           & \text{if $\xi_d \neq 0$,} \\
                      \bigl(\Fcal u\bigr)(\xi)  & \text{if $\xi_d = 0$,}
                    \end{cases}
\]
so that by means of Fourier inversion,
\[
  u = \Fcal^{-1}\bigl(\Fcal u\bigr)^{(1)}+ \Fcal^{-1}\bigl(\Fcal u\bigr)^{(2)}=: u^{(1)} + u^{(2)}.
\]
With this definition, $u^{(1)}$ and $u^{(2)}$ are $L^2(\R^d;\R^n)$-functions satisfying $\Fcal u^{(i)}=\bigl(\Fcal u\bigr)^{(i)}$ for $i=1,2$, and 
we may conclude from $\eqref{eq:A0-free}$ that
\begin{align*}
 \Acal_0 u^{(1)} = \Acal_0 u^{(2)} = 0 \qquad\text{in $\R^d$.}
\end{align*}
Turning to $u^{(2)}$ we observe that by construction $\partial_d u^{(2)} = 0$. In order to have the quadratic integrability of 
$u^{(2)}$ on $\R^d$ preserved it needs to hold that $u^{(2)}\equiv0$. Hence, for the proof of this proposition it will be enough 
to show the existence of an $\Acal_\eps$-free sequence $(u_\eps)_{\eps}$ with $u_\eps \to u^{(1)}$ in $L^2(\R^d;\R^n)$ as $\eps\to 0^+$.

For this purpose we set $u_\eps := \Pcal_{\Acal_\eps} u^{(1)}$ for every $\eps>0$ or, speaking in terms of Fourier transforms,
\[
  \Fcal u_\eps := \Pbb_{\Acal_\eps}(\frarg)\Fcal u^{(1)}
\]
with $\Pcal_{\Acal_\eps}$ and $\Pbb_{\Acal_\eps}$ as in Lemma~\ref{theo:projection}.
In view of $\eqref{def:A_tilde}$ and $\eqref{eq:A0-free}$ one finds for a.e.\ $\xi \in \R^d$ that
$\tilde{\Abb}_0(\xi) \bigl(\Fcal u^{(1)}\bigr)(\xi)= 0$,
which implies 
\[  \tilde{\Pbb}_0(\xi) \bigl(\Fcal u^{(1)}\bigr)(\xi)
  = \bigl(\Fcal u^{(1)}\bigr)(\xi).\]
Recalling that $\Fcal$ (and $\Fcal^{-1}$) are $L^2$-isometries we may argue that
\begin{align}\label{equation}
\norm{u_\eps-u^{(1)}}^2_{L^2(\R^d;\R^n)}&=\norm{\Fcal u_\eps-\Fcal u^{(1)}}^2_{L^2(\R^d;\R^n)}\nonumber\\
&=\int_{\R^d}\absb{\bigl(\Pbb_{\Acal_\eps}(\xi)-\tilde{\Pbb}_0(\xi)\bigr)\bigl(\Fcal u^{(1)}\bigr)(\xi)}^2\dd{\xi}.
\end{align}
At this point we apply Lemma~\ref{lem:P_eps_conv}.
This, together with the uniform boundedness of the projection operators $\Pbb_{\Acal_\eps}$, allows us to use 
Lebesgue's Dominated Convergence Theorem and we conclude that the right-hand side in $\eqref{equation}$ tends to zero for $\eps\to 0^+$.
\end{proof}

As mentioned at the beginning of this section, $\Acal^{\text{\rm mag}}$ is of the form $\eqref{operator_A}$ 
meeting the constant-rank condition and one can check easily that both the rank of $(A^{\text{\rm mag}})^{(3)}$ and its number of non-zero rows is 
three. Besides, a straightforward calculation shows that $\Acal^{\text{\rm mag}}_0$ as in the statement 
of Theorem~\ref{theo:mainresult_Gamma} corresponds to $\Acal^{\text{\rm mag}}$ in the sense of Proposition~\ref{prop:recov_1}.
Hence, the next corollary is an immediate consequence of the previous proposition.
\begin{corollary}\label{cor:recov_1}
For every $(m,h)\in \Ucal_0$ there exists a sequence 
\begin{align*}
(\hat{m}_\eps, \hat{h}_\eps)_\eps\subset L^2(\R^3;\R^3)\times L^2(\R^3;\R^3) 
\end{align*}
such that
$(\hat{m}_\eps, \hat{h}_\eps)\to (m,h)$ in $L^{2}(\R^3;\R^3)\times L^2(\R^3;\R^3)$ as $\eps\to 0^+$ and
\begin{align*}\Acal_\eps^{\text{\rm mag}} \Bigl(\begin{array}{c}
    \hat{m}_\eps \\
    \hat{h}_\eps 
\end{array}\Bigr)=0 \end{align*} 
for all $\eps>0$.
\end{corollary}

Before we prove the upper bound, we establish the following relation between $\curl_\eps$- and $\divergerg_\eps$-free fields,
which is a type of Helmholtz decomposition for $L^2$-functions defined on $\R^3$.
\begin{lemma}\label{lem:projection_divcurl}
For every $\eps>0$ it holds that 
\begin{align*}
\Ical-\Pcal_{\curl_\eps}= \Pcal_{\divergerg_\eps},
\end{align*}
where $\Pcal_{\curl_\eps}, \Pcal_{\divergerg_\eps}:L^2(\R^3;\R^3)\to L^2(\R^3;\R^3)$ are projection operators as defined in 
Lemma~\ref{theo:projection} and $\Ical$ is the identity map on $L^2(\R^3;\R^3)$.
\end{lemma}

\begin{proof}
In view of $\eqref{def:projector_P}$, the claim holds true if 
\begin{align*}
\Ibb-\Pbb_{\curl_\eps}(\xi)= \Pbb_{\divergerg_\eps}(\xi) \qquad \text{ for all $\xi\in \R^3\setminus\{0\}$}.
\end{align*}
with $\Ibb:\R^3\to\R^3$ the identity function. In the sequel we use the notation $\xi_\eps=(\xi_1, \xi_2, 1/\eps\,\xi_3)^T$ for $\xi\in \R^3$. 
Let $\xi\in \R^3\setminus\{0\}$ be fixed. Then 
\begin{align*}
 \ker \Abb_{\divergerg_\eps}(\xi)=\set{v\in \R^3}{\xi_\eps\cdot v =0}=\Span\{\xi_\eps^{\perp,1}, \xi_\eps^{\perp,2}\}
\end{align*}
with orthogonal unit vectors $\xi_\eps^{\perp,1}$ and $\xi_\eps^{\perp,2}$, while
\begin{align*}
 \ker \Abb_{\curl_\eps}(\xi)=\set{v\in \R^3}{\xi_\eps\times v =0}=\Span\{\xi_\eps/\abs{\xi_\eps}\}.
\end{align*}
Thus, for $v\in \R^3$ one may infer 
\begin{align*}
\bigl(\Ibb-\Pbb_{\curl_\eps}(\xi)\bigr)v= v-(v\cdot\xi_\eps)\xi_\eps/\abs{\xi_\eps}^2=(v\cdot\xi_\eps^{\perp,1})\xi_\eps^{\perp,1}+(v\cdot\xi_\eps^{\perp,2})\xi_\eps^{\perp,2}= \Pbb_{\divergerg_\eps}(\xi)v
\end{align*}
and the proof is complete.
\end{proof}

Let us point out that the functions $\hat{m}_\eps$ in Corollary~\ref{cor:recov_1} do not have $W^{1,2}$-regularity nor 
do they fulfill the required nonconvex constraint of local saturation. Consequently, $(\hat{m}_\eps, \hat{h}_\eps)_\eps$ 
fails to be a correct recovery sequence. In order to overcome this problem we make a construction based on the use of 
appropriate projection operators and prove the following proposition.

\begin{proposition}[Recovery sequence]
For every $(m,h)\in \Ucal_0$ there exists a sequence $(m_\eps, h_\eps)_\eps\subset \Ucal_\eps$ with 
$(m_\eps, h_\eps)\to (m,h)$ in $W^{1,2}(\Omega_1;\R^3)\times L^2(\R^3;\R^3)$ satisfying
\begin{align*}
 \lim_{\eps\to 0^+}\, F_\eps[m_\eps,h_\eps]= F_0[m,h].
\end{align*}
\end{proposition}

\begin{proof}
For given $(m,h)\in \Ucal_0$ let $(\hat{m}_\eps, \hat{h}_\eps)_\eps$ be as in Corollary~\ref{cor:recov_1}. We set for $\eps>0$,
\begin{align*}
m_\eps & =m,\\
h_\eps & = \Pcal_{\curl_\eps}\bigl(\hat{h}_\eps-m+\hat{m}_\eps\bigr).
\end{align*}
The assertion is that this definition of $(m_\eps, h_\eps)_\eps$ provides a recovery sequence for $(m,h)$. 

Indeed, it holds that $\Acal_\eps^{\text{\rm mag}}(m_\eps, h_\eps)=0$ in $\R^3$ for every $\eps>0$, 
since $\curl_\eps h_\eps =0$ by Lemma~\ref{theo:projection}\,(ii) and 
\begin{align*}
\divergerg_\eps(m_\eps+ h_\eps) &= \divergerg_\eps\bigl(m + \Pcal_{\curl_\eps}(\hat{h}_\eps-m+\hat{m}_\eps)\bigr)\\ 
&=\divergerg_\eps(\hat{m}_\eps+ \hat{h}_\eps) + \divergerg_\eps\bigl((\Pcal_{\curl_\eps}- \Ical)(\hat{h}_\eps-m+\hat{m}_\eps)\bigr)=0
\end{align*}
in view of Lemma~\ref{lem:projection_divcurl} and Corollary~\ref{cor:recov_1}. Trivially, $m_\eps$ satisfies 
$\abs{m_\eps}=\abs{m}=m_s$ in $\Omega_1$ for all $\eps>0$.
Moreover, applying Lemma~\ref{theo:projection}\,(iii)-(iv) yields
\begin{align*}
&\norm{h_\eps-h}_{L^2(\R^3;\R^3)} \leq\norm{h_\eps-\hat{h}_\eps}_{L^2(\R^3;\R^3)}+\norm{\hat{h}_\eps-h}_{L^2(\R^3;\R^3)}\\
&\qquad\leq \norm{\Pcal_{\curl_\eps}\hat{h}_\eps-\hat{h}_\eps}_{L^2(\R^3;\R^3)} + \norm{\Pcal_{\curl_\eps}(\hat{m}_\eps-m)}_{L^2(\R^3;\R^3)} 
+\norm{\hat{h}_\eps-h}_{L^2(\R^3;\R^3)}\\
&\qquad\leq C\, \bigl(\norm{\curl_\eps \hat{h}_\eps}_{W^{-1,2}(\R^3;\R^3)} + \norm{\hat{m}_\eps-m}_{L^2(\R^3;\R^3)}\bigr) 
+ \norm{\hat{h}_\eps-h}_{L^2(\R^3;\R^3)}\\
&\qquad= C\,\norm{\hat{m}_\eps-m}_{L^2(\R^3;\R^3)} + \norm{\hat{h}_\eps-h}_{L^2(\R^3;\R^3)}.
\end{align*}
This expression tends to zero by Corollary~\ref{cor:recov_1} as $\eps\to 0^+$.

Summarizing, we find that $(m_\eps,h_\eps)\in \Ucal_\eps$ for all $\eps>0$ and 
\begin{align*}
&\lim_{\eps\to 0^+} \int_{\Omega_1} \alpha\, \abs{\nabla_\eps m_\eps}^2 + \varphi(m_\eps)\dd{x} + \frac{1}{2} \int_{\R^3}\abs{h_\eps}^2 \dd{x} \\
&\qquad\qquad\qquad\qquad\qquad\qquad\qquad\qquad\quad=\int_{\Omega_1} \alpha\, \abs{\nabla' m}^2 
+ \varphi(m)\dd{x} + \frac{1}{2} \int_{\R^3}\abs{h}^2 \dd{x}.
\end{align*}
\end{proof}


\section{Comparison with the results by Gioia and James}\label{sec:comparisonGJ97}
To demonstrate the agreement between Theorem~\ref{theo:mainresult_Gamma} and the result obtained in \cite[Theorem~4.1]{GJ97} we will prove that the 
$\Gamma$-limit of $F_\eps$ as $\eps\to 0^+$ can be expressed equivalently in the form
\begin{align*}
\tilde{F}_0[m, h]=\left\{\begin{array}{cl}\displaystyle\int_{\omega} \alpha\, \abs{\nabla m}^2 
+ \varphi(m)+ \frac{1}{2} m_3^2\dd{x}, &\text{if $(m,h)\in \tilde{\Ucal}_0$},\\
\infty, & \text{otherwise\,,}\end{array}\right.
\end{align*}
where
\begin{align*}
\tilde{\Ucal}_0=\set{(m, h)\in W^{1,2}(\omega;\R^3)\times L^2(\R^2;\R^3)}{\abs{m}=m_s \text{ in $\omega$},\; 
h=-(0,\,0, m_3)^T}.
\end{align*}
Indeed, if we identify $m$ and $h$ with their constant extensions in space direction $x_3$, it is sufficient to show that
\begin{align*}
\Ucal_0=\tilde{\Ucal}_0.
\end{align*}
This equality results essentially from the observation that the only function that is constant with respect to one of the coordinate directions 
and at the same time quadratically integrable on the whole space is the zero mapping. So for $(m,h)\in \Ucal_0$ one finds 
$\partial_3 h_2=\partial_3 h_1=0$ in $\R^3$ and $\partial_3 h_3=0$ in $\R^3\setminus \Omega_1$, which results in $h_1\equiv h_2\equiv 0$ and 
$h_3\equiv0$ in $\R^3\setminus\Omega_1$, respectively. Applying the same argument once again to $\partial_3\bigl((m_3+h_3)\chi_{\Omega_1}\bigr)=0$ in 
$\R^3$, where $\chi_{\Omega_1}$ is the indicator function of $\Omega_1$, yields $h_3=-m_3\chi_{\Omega_1}=-m_3$. Finally, since $m$ is independent 
of the $x_3$-variable, $(m,h)$ is actually two-dimensional and we conclude $(m,h)\in \tilde{\Ucal}_0$. The reverse inclusion follows simply from the 
above mentioned identification and a straightforward calculation.

\begin{remark}
a) In view of its representation $\tilde{F}_0$, the $\Gamma$-limit of $F_\eps$ is purely two-dimensional. The third additive term in the 
density of $\tilde{F}_0$ indicates that magnetizations pointing out of plane are penalized and hence less favorable when it comes 
to energy minimization. Interestingly, we also observe that the limit functional $\tilde{F}_0$ can be viewed as local. In fact, the Maxwell equations 
disappear in the limit $\eps\to 0^+$, so that the asymptotic problem is free of any magnetostatic constraints of this form. Instead, the 
relation between $m$ and $h$ is ruled by the simple pointwise equality $h=-(0,0, m_3)^T$.
For a detailed interpretation of our $\Gamma$-convergence result with respect to physical and engineering applications we refer to \cite[Section~5]{GJ97}. 
There, Gioia and James analyze the qualitative effect of external fields and the practical scope of the theory regarding the thickness of 
films by giving precise estimates for some relevant materials. 

b) The main advantage of our approach is the convenient access to complete 
information about the induced field $h$, which is automatically included in $\tilde{F}_0$. In contrast, the reasoning in~\cite{GJ97} 
requires explicit solving of the magnetostatic equations together with a limit analysis for the solutions as tends to zero to gain the same insight.

c) The scaling of the micromagnetic free energy in this work rests fundamentally on the assumption
that the material parameter $\alpha$ does not depend on the thickness $\eps$. 
We refer to \cite{DKMO_Review} for a review paper discussing a number of different reduction regimes. Here the characteristic length 
scale $d$ ($\simeq \sqrt{\alpha}$) of the magnetic material is related to $\eps$ and $l$ (the length scale of the cross section) in
such a way that $d^2/l\eps \to 0$ and $\eps/l\to 0$. In \cite{DKMO02}, for instance, the authors study the limit behavior in
the regime $l\eps/d^2 \gg \ln(l/\eps)$ and are able to capture domain structures observed in large thin films,
as comparison with experimental data underlines. 
Let us point out once again that the results obtained in \cite{GJ97}
and recovered here are physically relevant only for thin-film samples of sufficiently small lateral expansion. In this case one observes 
a single region of uniform magnetization.
\end{remark}


\section*{Acknowledgments}
\sloppypar 
I am grateful to Irene Fonseca for pointing me to this topic, for valuable conversations on the subject and 
for reading carefully a first draft of the manuscript. Also, my thanks go to Filip Rindler, who contributed with 
useful ideas on closely related issues. This research was carried out during a one-year stay at Carnegie Mellon 
University funded by the Funda\c{c}\~{a}o para a Ci\^{e}ncia e a Tecnologia (FCT) through the ICTI CMU--Portugal 
program and UTA-CMU/MAT/0005/2009.


\bibliographystyle{amsplain}

\end{document}